\documentclass[12pt]{amsart}
\usepackage{hyperref,cleveref,amssymb}

\theoremstyle{plain}
\newtheorem{theorem}{Theorem}[section]

\theoremstyle{definition}

\newcommand{\abs}[1]{\lvert#1\rvert}
\newcommand{\norm}[1]{\lVert#1\rVert}
\newcommand{\bigabs}[1]{\bigl\lvert#1\bigr\rvert}
\newcommand{\bignorm}[1]{\bigl\lVert#1\bigr\rVert}

\renewcommand{\mid}{\::\:}

\makeatletter
\def\@tvsp{\mathchoice{{}\mkern-4.5mu}{{}\mkern-4.5mu}{{}\mkern-2.5mu}{}}
\def\ltrivert{\left|\@tvsp\left|\@tvsp\left|}
\def\rtrivert{\right|\@tvsp\right|\@tvsp\right|}
\def\bltrivert{\bigl|\@tvsp\bigl|\@tvsp\bigl|}
\def\brtrivert{\bigr|\@tvsp\bigr|\@tvsp\bigr|}
\makeatother
\newcommand{\trinorm}[1]{\ltrivert#1\rtrivert}

\DeclareMathOperator{\supp}{supp}
\DeclareMathOperator{\FVL}{FVL}
\DeclareMathOperator{\FBL}{FBL}

\renewcommand{\le}{\leqslant}
\renewcommand{\ge}{\geqslant}

\begin{document}

\title{Simple constructions of $\mathrm{FBL}(A)$ and $\mathrm{FBL}[E]$}

\author{V.G. Troitsky}
\address{Department of Mathematical and Statistical Sciences,
         University of Alberta, Edmonton, AB, T6G\,2G1, Canada.}
\email{troitsky@ualberta.ca}

\thanks{The author was supported by an NSERC grant.}

\date{\today}

\begin{abstract}
  We show that the free Banach lattice $\mathrm{FBL}(A)$ may be
  constructed as the completion of $\mathrm{FVL}(A)$ with respect to
  the maximal lattice seminorm $\nu$ on $\mathrm{FVL}(A)$ with
  $\nu(a)\le 1$ for all $a\in A$. We present a similar construction
  for the free Banach lattice $\mathrm{FBL}[E]$ generated by a Banach
  space $E$.
\end{abstract}

\maketitle

\section{Preliminaries}

The free vector lattice over a set $A$, denoted by $\FVL(A)$, goes
back to~\cite{Birkhoff:42}. More recently, a free Banach lattice
$\FBL(A)$ has been introduced and investigated;
see~\cite{dePagter:15,Aviles:18}. It has been folklore knowledge (and
was implicitly mentioned in \cite{dePagter:15,Aviles:18}) that the
norm of $\FBL(A)$ is, in some sense, the greatest lattice norm one can
put on $\FVL(A)$. In this note, we make this idea into a formal
statement and provide a direct proof. This yields an alternative way
of constructing $\FBL(A)$ and $\FBL[E]$.

Let $A$ be a subset of a vector lattice~$X$. We say that $X$ is a free
vector lattice over $A$ if every function $\varphi\colon A\to Y$,
where $Y$ is an arbitrary vector lattice, extends uniquely to a
lattice homomorphism $\tilde\varphi\colon X\to Y$. For every set $A$
there is a vector lattice $X$ which contains $A$ and is free
over~$A$. It is easy to see that if $X_1$ and $X_2$ are both free over
$A$ then there exists a lattice isomorphism between $X_1$ and $X_2$
which fixes~$A$. So a free vector lattice over $X$ is determined
uniquely up to a lattice isomorphism; we denote it by $\FVL(A)$.

We outline below a construction of $\FVL(A)$ and some of its basic
properties; we refer the reader to~\cite{Bleier:73,dePagter:15} for
further details on free vector lattices. Given a set~$A$.  For every
$a\in A$, we write $\delta_a$ for the ``evaluation functional'' of $a$
in the following sense: $\delta_a\colon\mathbb R^A\to\mathbb R$ with
$\delta_a(x)=x(a)$ for $x\colon A\to\mathbb R$. Then $\FVL(A)$ may be
identified with the sublattice of $\mathbb R^{{\mathbb R}^A}$
generated by $\{\delta_a\mid a\in A\}$. Identifying $a$
with~$\delta_a$, one may view $A$ as a subset of $\FVL(A)$. Since
$\FVL(A)$ is a sublattice $\mathbb R^{{\mathbb R}^A}$, and the latter
is Archimedean, $\FVL(A)$ is also Archimedean.

It is easy to see that if $a_1,\dots,a_n\in A$ then
$\FVL\bigl(\{a_1,\dots,a_n\}\bigr)$ may be identified with the
sublattice of $\FVL(A)$ generated by $a_1,\dots,a_n$. For every
$f\in\FVL(A)$ there exists a finite subset $\{a_1,\dots,a_n\}$ of $A$
such that $f$ belongs to the sublattice of $\FVL(A)$ generated by
$a_1,\dots,a_n$.  Furthermore, if $A$ itself is finite, say,
$A=\{a_1,\dots,a_n\}$, then $\bigvee_{k=1}^n\abs{a_k}$ is a strong
unit in $\FVL(A)$.

By a lattice-linear expression, we mean an expression formed by
finitely many variables and linear and lattice operations.  For
example, $F(t_1,t_2,t_3)=t_1\wedge t_2+t_1\vee(2t_3)$ is a
lattice-linear expression. Clearly, a lattice-linear expression
$F(t_1,\dots,t_n)$ induces a positively homogeneous function from
$\mathbb R^n$ to~$\mathbb R$. On the other hand, if $X$ is an
Archimedean vector lattice and $x_1,\dots,x_n\in X$, plugging
$x_1,\dots,x_n$ into $F$ instead of $t_1,\dots,t_n$, we can define
$F(x_1,\dots,x_n)$ as an element of $X$ in a natural way. We say that
$F(x_1,\dots,x_n)$ is a lattice-linear combination of $x_1,\dots,x_n$.
If two lattice-linear expressions $F$ and $G$ agree as functions from
$\mathbb R^n$ to $\mathbb R$ then
$F(x_1,\dots,x_n)=G(x_1,\dots,x_n)$. Actually, the calculus of
lattice-linear expressions in $X$ is a restriction of Krivine's
function calculus; see, e.g., \cite[Proposition~3.6]{Buskes:91}.
Observe that the sublattice of $X$ generated by $x_1,\dots,x_n$ is
exactly the set of all lattice-linear combinations of
$x_1,\dots,x_n$. $\FVL(A)$ may be interpreted as the set of all formal
lattice-linear expressions of elements of~$A$, where we identify two
expressions if they agree as functions from $\mathbb R^n$ to
$\mathbb R$. For example, we identify $a_1+(a_2\vee a_3)$ and
$(a_1+a_2)\vee(a_1+a_3)$. Formally speaking, $\FVL(A)$ consists of
equivalence classes of lattice-linear expressions.

In~\cite{dePagter:15}, the concept of a free Banach lattice was
introduced. Let $A$ be a subset of a Banach lattice~$X$. We say that
$X$ is a free Banach lattice over a set $A$ if every function
$\varphi\colon A\to Y$, where $Y$ is an arbitrary Banach lattice,
satisfying $\sup_{a\in A}\bignorm{\varphi(a)}\le 1$ extends uniquely
to a lattice homomorphism $\tilde\varphi\colon X\to Y$ with
$\norm{\tilde\varphi}\le 1$. It was shown in~\cite{dePagter:15} that
for every set $A$ there is a Banach lattice $X$ which contains $A$ and
is free over~$A$. Again, it is easy to see that such a Banach lattice
is unique up to a lattice isometry which fixes $A$; we denote it by
$\FBL(A)$.  It is easy to see that $A$ is a subset of the unit sphere
of $\FBL(A)$.

An alternative way of constructing $\FBL(A)$ was recently obtained
in~\cite{Aviles:18}. In~\cite{Aviles:18}, the authors also prove that
for every Banach space $E$ there exists a Banach lattice $\FBL[E]$
such that $E$ is a closed subspace of $\FBL[E]$ and every bounded
operator $T\colon E\to Y$, where $Y$ is an arbitrary Banach lattice,
extends uniquely to a lattice homomorphism
$\widetilde{T}\colon\FBL[E]\to Y$ with
$\norm{\widetilde{T}}=\norm{T}$. It is easy to see that $\FBL[E]$ is
unique up to a lattice isometry preserving~$E$. Furthermore, it can be
easily verified that $\FBL(A)=\FBL\bigl[\ell_1(A)]\bigr]$ for any set
$A$.

In this note, we present constructions of $\FBL(A)$ and $\FBL[E]$
that are somewhat easier than those in~\cite{dePagter:15,Aviles:18}.

\section{A construction of $\FBL(A)$}
\label{sec:FBLset}

\begin{theorem}
  There exists a maximal lattice seminorm $\nu$ on $\mathrm{FVL}(A)$
  with $\nu(a)\le 1$ for all $a\in A$. It is a lattice norm, and the
  completion of $\FVL(A)$ with respect to it is $\mathrm{FBL}(A)$.
\end{theorem}

\begin{proof}
As before, we
identify $\FVL(A)$ with the sublattice of $\mathbb R^{{\mathbb R}^A}$
generated by $\{\delta_a\mid a\in A\}$; by identifying $a\in A$ with
$\delta_a\in\FVL(A)$, we may view $A$ as a subset of $\FVL(A)$. Let
$\mathcal N$ be the set of all lattice seminorms $\nu$ on $\FVL(A)$
such that $\nu(\delta_a)\le 1$ for every $a\in A$.

Let $x\in\mathbb R^A$ such that $\abs{x(a)}\le 1$ for all $a\in A$. For
$f\in\FVL(A)$, put $\nu_x(f)=\bigabs{f(x)}$. It can be easily verified
that $\nu_x\in\mathcal N$.

For $f\in\FVL(A)$, put
\begin{math}
  \norm{f}=\sup\limits_{\nu\in\mathcal N}\nu(f).
\end{math}
We claim that this is a lattice norm on $\FVL(A)$.

First, observe that $\norm{f}$ is finite. Find $a_1,\dots,a_n\in A$
such that $f\in\FVL\bigl(\{a_1,\dots,a_n\}\bigr)$. Since
\begin{math}
  \abs{\delta_{a_1}}+\dots+\abs{\delta_{a_n}}
\end{math}
is a strong unit in $\FVL\bigl(\{a_1,\dots,a_n\}\bigr)$, there exists
$\lambda\in\mathbb R_+$ such that
\begin{displaymath}
  \abs{f}\le\lambda\bigl(\abs{\delta_{a_1}}+\dots+\abs{\delta_{a_n}}\bigr).
\end{displaymath}
It follows that
\begin{math}
  \nu(f)\le\lambda n
\end{math}
for every $\nu\in\mathcal N$, hence $\norm{f}\le\lambda n<\infty$.

It is straightforward that $\norm{\cdot}$ is positively
homogeneous. To verify the triangle inequality, let
$f,g\in\FVL(A)$ and fix $\varepsilon>0$. There exists
$\nu\in\mathcal N$ such that
\begin{displaymath}
  \norm{f+g}-\varepsilon<\nu(f+g)\le\nu(f)+\nu(g)
  \le\norm{f}+\norm{g}.
\end{displaymath}
It follows that $\norm{f+g}\le\norm{f}+\norm{g}$.

Suppose that $f\ne 0$. Find $a_1,\dots,a_n\in A$ such that
$f\in\FVL\bigl(\{a_1,\dots,a_n\}\bigr)$. It follows that there is
$x\in\mathbb R^A$ such that $f(x)\ne 0$ and
$\supp x\subseteq\{a_1,\dots,a_n\}$. Without loss of generality,
scaling $x$ if necessary, $\abs{x(a)}\le 1$ for all $a\in A$. Then
$\nu_x(f)=\bigabs{f(x)}> 0$, hence $\norm{f}\ne 0$.

If $\abs{f}\le\abs{g}$ in $\FVL(A)$ then $\nu(f)\le\nu(g)$ for every
$\nu\in\mathcal N$, hence $\norm{f}\le\norm{g}$. Thus, $\norm{\cdot}$
is a lattice norm on $\FVL(A)$.

Let $X$ be the completion of $\bigl(\FVL(A),\norm{\cdot}\bigr)$. We
claim that $X$ is a free Banach lattice over~$A$.

Let $\varphi\colon A\to Y$, where $Y$ is a Banach lattice and
$\sup_{a\in A}\norm{\varphi(a)}\le 1$. Then $\varphi$ extends
to a lattice homomorphism $\hat\varphi\colon\FVL(A)\to Y$. It
suffices to show that $\norm{\hat\varphi}\le 1$; it would follow
that $\hat\varphi$ extends to a contractive lattice homomorphism
$\tilde\varphi$ from $X$ to~$Y$. For $f\in\FVL(A)$, put
$\nu(f)=\bignorm{\hat\varphi(f)}$. It is easy to see that
$\nu\in\mathcal N$, hence
$\bignorm{\hat\varphi(f)}=\nu(f)\le\norm{f}$.

Uniqueness of the extension follows from the fact that any lattice
homomorphism extension of $\varphi$ to $X$ has to agree with
$\hat\varphi$ on $\FVL(A)$, hence with $\tilde\varphi$ on $X$ as
$\FVL(A)$ is dense in~$X$.
\end{proof}

It follows that the FBL norm is the greatest norm on $\FVL(A)$ such
that $\norm{a}\le 1$ for every $a\in A$.

\section{A construction of $\FBL[E]$}
\label{sec:FBLspace}

For a Banach lattice $E$ and a vector $x\in E$, the evaluation
functional $\hat x\in E^{**}$ is defined by $\hat x(x^*)=x^*(x)$ for
$x^*\in X^*$. In particular, $\hat x$ is a function from $E^*$
to~$\mathbb R$, i.e., an element of $\mathbb R^{E^*}$.

\begin{theorem}
  Let $E$ be a Banach space; let $L$ be the sublattice of
  $\mathbb R^{E^*}$ generated by $\{\hat x\mid x\in E\}$. There is a
  maximal lattice seminorm $\nu$ on $L$ satisfying
  $\nu(\hat x)\le\norm{x}$ for all $x\in E$. It is a lattice norm; the
  completion of $L$ with respect to it is $\mathrm{FBL}[E]$.
\end{theorem}

\begin{proof}  
  It is easy to see that the map $x\in E\mapsto\hat x\in L$ is a
  linear embedding, so that we may view $E$ as a linear subspace
  of~$L$. Let $\mathcal M$ be the set of all lattice seminorms $\nu$
  on $L$ such that $\nu(\hat x)\le\norm{x}$ for all $x\in E$.

For every $x^*\in B_{E^*}$, define $\nu_{x^*}(f)=\bigabs{f(x^*)}$ for
$f\in L$. It is easy to see that $\nu_{x^*}$ is a seminorm
on~$L$. Note that $\nu_{x^*}$ is a lattice seminorm because
$\abs{f}(x^*)=\bigabs{f(x^*)}$. Furthermore, if $x\in E$ then
\begin{math}
  \nu_{x^*}(\hat x)=\bigabs{\hat x(x^*)}=\bigabs{x^*(x)}\le\norm{x},
\end{math}
so that $\nu_{x^*}\in\mathcal M$. 

For $f\in L$, define $\trinorm{f}=\sup_{\nu\in\mathcal M}\nu(f)$. We
claim that $\trinorm{\cdot}$ is a lattice norm on~$L$. First, we will
show that it is finite. Let $f\in L$. Then $f$ is a lattice-linear
expression of $\hat x_1,\dots,\hat x_n$ for some $x_1,\dots,x_n$
in~$E$. Since lattice-linear functions are positively homogeneous, we
may assume without loss of generality that $x_1,\dots,x_n\in
B_E$. Clearly, $\abs{\hat x_1}+\dots+\abs{\hat x_n}$ is a strong unit
in the sublattice of $L$ generated by $\hat x_1,\dots,\hat x_n$. It
follows that
\begin{math}
  \abs{f}\le\lambda\bigl(\abs{\hat
  x_1}+\dots+\abs{\hat x_n}\bigr)
\end{math}
for some $\lambda>0$, so that $\nu(f)\le\lambda n$ for every
$\nu\in\mathcal M$, hence $\trinorm{f}\le\lambda n<\infty$.

It is straightforward that $\trinorm{\cdot}$ is a lattice seminorm
on~$L$. Suppose that $0\ne f\in L$. Then $f(x^*)\ne 0$ for some
$x^*\in E^*$. Since $f$ is positively homogeneous, we may assume
without loss of generality that $x^*\in B_{E^*}$. Then
$\trinorm{f}\ge \nu_{x^*}(f)=\bigabs{f(x^*)}>0$.  Thus,
$\trinorm{\cdot}$ is a lattice norm on~$L$.

Note also that $\trinorm{\hat x}=\norm{x}$ for all $x\in E$, so that
we may view $\trinorm{\cdot}$ as an extension of $\norm{\cdot}$ from
$E$ to~$L$. Indeed, $\nu(\hat x)\le\norm{x}$ for every $\nu\in\mathcal
M$, hence $\trinorm{\hat x}\le\norm{x}$. On the other hand, let
$x^*\in B_{E^*}$ such that $x^*(x)=\norm{x}$; then
\begin{math}
  \trinorm{\hat x}\ge\nu_{x^*}(\hat x)=\bigabs{\hat x(x^*)}=\norm{x}.
\end{math}

Let $X$ be the completion of $\bigl(L,\trinorm{\cdot}\bigr)$. We claim
that $X=\FBL[E]$.

Let $T\colon E\to Y$ be a linear operator from $E$ to an arbitrary
Banach lattice $Y$ with $\norm{T}=1$. We define
$\widehat{T}\colon L\to Y$ as follows. Let $f\in L$. Then $f$ is a
lattice-linear combination of $\hat x_1,\dots,\hat x_n$ for some
$x_1,\dots,x_n\in E$. Without loss of generality, $x_1,\dots,x_n$ are
linearly independent in~$E$. We define $\widehat{T}f$ to be the same
lattice-linear combination of $Tx_1,\dots,Tx_n$ in~$Y$. That is,
suppose that $f=F(\hat x_1,\dots,\hat x_n)$ for some formal
lattice-linear expression $F(t_1,\dots,t_n)$; we then put
$\widehat{T}f=F(Tx_1,\dots,Tx_n)$.  Note that $\widehat{T}f$ is
well-defined, i.e., does not depend on a particular choice of a
lattice-linear combination representing~$f$. Indeed, suppose that
$f=G(\hat x_1,\dots,\hat x_n)$, where $G(t_1,\dots,t_n)$ is another
formal lattice-linear expressions. Since $L$ is a sublattice of
$\mathbb R^{E^*}$, the lattice operations in $L$ are point-wise, hence
\begin{displaymath}
  f(x^*)
  =F\bigl(\hat x_1(x^*),\dots,\hat x_n(x^*)\bigr)
  =F\bigl(x^*(x_1),\dots,x^*(x_n)\bigr)
\end{displaymath}
in $\mathbb R$ for every $x^*\in X^*$. Similarly,
\begin{math}
  f(x^*)=G\bigl(x^*(x_1),\dots,x^*(x_n)\bigr).
\end{math}
Since
$x_1,\dots,x_n$ are linearly independent, this means that
\begin{math}
  F(t_1,\dots,t_n)=G(t_1,\dots,t_n)
\end{math}
for all $t_1,\dots,t_n\in\mathbb R$.
Therefore, $\widehat{T}$ is well-defined.

The definition of $\widehat{T}$ immediately yields that it is a
lattice homomorphism. Clearly, $\widehat{T}$ extends $T$ in the sense
that $\widehat{T}\hat x=Tx$ for every $x\in E$. It follows that
$\norm{\widehat{T}}\ge 1$. We claim that that
$\norm{\widehat{T}}=1$. Indeed, for $f\in L$, define
$\nu(f)=\norm{\widehat{T}f}$ in~$Y$. It is easy to see that $\nu$ is a
lattice seminorm on~$L$. For every $x\in X$, one has
\begin{math}
  \nu(\hat x)=\norm{\widehat{T}\hat x}=\norm{Tx}\le\norm{x}.
\end{math}
It follows that $\nu\in\mathcal M$, so that
\begin{math}
  \norm{\widehat{T}f}=\nu(f)\le\trinorm{f},
\end{math}
so that $\norm{\widehat{T}}\le 1$. It follows that $\widehat{T}$
extends to a contractive lattice homomorphism $\widetilde{T}\colon
X\to Y$.

Again, uniqueness of the extension follows from the fact that any
contractive lattice homomorphism that extends $T$ to $X$ has to agree
with $\widehat{T}$ on $L$ and, therefore, with $\widetilde{T}$ on~$X$.
\end{proof}

\end{document}